\documentclass[12pt]{amsart}

 \usepackage{graphicx,color}
\usepackage{geometry}
\usepackage{color, graphicx}
\usepackage[all]{xy}

\usepackage{calrsfs}
\usepackage{times}
\usepackage[scaled=0.92]{helvet}

\theoremstyle{plain}
\swapnumbers
\newtheorem{thm}{Theorem}[section]
\newtheorem{prop}[thm]{Proposition}
\newtheorem{cor}[thm]{Corollary}

\theoremstyle{definition}
\newtheorem{df}[thm]{Definition}
\newtheorem{notation}[thm]{Notation}
\newtheorem{example}[thm]{Example}

\newtheorem{rem}[thm]{Remark}
\newtheorem{remarks}[thm]{Remarks}

\usepackage{amsfonts}
\usepackage{amssymb}

\renewcommand{\bar}{\overline}
\DeclareMathOperator{\Z}{\mathbf{Z}}
\DeclareMathOperator{\bZ}{\mathbf{Z}}
\DeclareMathOperator{\cO}{\mathcal{O}}
\DeclareMathOperator{\Q}{\mathbf{Q}}
\DeclareMathOperator{\C}{\mathbf{C}}
\DeclareMathOperator{\F}{\mathbf{F}}
\DeclareMathOperator{\bF}{\mathbf{F}}
\DeclareMathOperator{\P1}{\mathbf{P}}
\DeclareMathOperator{\Gal}{\mathrm{Gal}}
\DeclareMathOperator{\Ind}{\mathrm{Ind}}
\DeclareMathOperator{\p}{\mathfrak{p}}
\DeclareMathOperator{\B}{\mathfrak{P}}
\DeclareMathOperator{\m}{\mathfrak{m}}
\DeclareMathOperator{\PSL}{\mathrm{PSL}}

\begin{document}

\date{\today\ (version 1.0)} 
\title[Arithmetic equivalence and the Goss zeta function]{Arithmetic equivalence for function fields, \\[1mm] the Goss zeta function and a generalization}
\author[G.~Cornelissen]{Gunther Cornelissen}
\address{Mathematisch Instituut, Universiteit Utrecht, Postbus 80.010, 3508 TA Utrecht, Nederland}
\email{g.cornelissen@uu.nl}
\author[A.~Kontogeorgis]{Aristides Kontogeorgis}
\address{Department of Mathematics, University of the Aegean, Karlovasi, Samos, Greece}
\email{kontogar@aegean.gr}
\author[L.~van der Zalm]{Lotte van der Zalm}
\address{Mathematisch Instituut, Universiteit Utrecht, Postbus 80.010, 3508 TA Utrecht, Nederland}
\email{l.zalm@planet.nl}


\begin{abstract} \noindent A theorem of Tate and Turner says that global function fields have the same zeta function if and only if the Jacobians of the corresponding curves are isogenous. In this note, we investigate what happens if we replace the usual (characteristic zero) zeta function by the positive characteristic zeta function introduced by Goss. We prove that for function fields whose characteristic exceeds their degree, equality of the  Goss zeta function is the same as Gassmann-equivalence (a purely group theoretical property), but this statement fails if the degree exceeds the characteristic. We introduce a `Teichm\"uller lift' of the Goss zeta function and show that equality of such is \emph{always} the same as Gassmann equivalence. \end{abstract}

\thanks{We thank Jakub Byszewski and Gautam Chinta  for useful discussions.}

\maketitle

\section{Introduction}

The following classical theorem about number fields shows a surprising equivalence between a number theoretical, a group theoretical, and an analytical statement:

\begin{thm} \label{thm1}
Let $K$ and $L$ denote two number fields. The following are equivalent:
\begin{enumerate}

\item[\textup{(i)}] \textup{[Arithmetic Equivalence]} For every prime number $p$ outside a set of Dirichlet density zero, the \emph{splitting type} of $p$ (i.e., the ordered sequence of inertia degrees with multplicities) in $K$ and $L$ is the same;
\item[\textup{(i')}] \textup{[Split Equivalence]} For every prime number $p$ outside a set of Dirichlet density zero, the number of prime ideals in $K$ and $L$ above $p$ are the same;
\item[\textup{(ii)}] \textup{[Ga{\ss}mann Equivalence]} Let $N$ denote a Galois extension of $\Q$ containing $K$ and $L$; then every conjugacy class in $\Gal(N/\Q)$ meets $\Gal(N/K)$ and $\Gal(N/L)$ in the same number of elements;  
\item[\textup{(iii)}] $K$ and $L$ have the same Dedekind zeta function: $\zeta_K(s)=\zeta_L(s)$.
\end{enumerate}
\end{thm}

For proofs, see \cite{Perlis1}, \cite{Perlis2} for the equivalence of (i) and (i'), \cite{Klingen} or the historical source \cite{G}. The theorem can be used to show that the zeta function of a number field does not necessarily determine the number field up to field isomorphism. We also remark that the statement in (iii) is sometimes called arithmetic equivalence (e.g., \cite{dS}), but since here, we will study various zeta functions and how the equivalences in the above theorem depend on what zeta function is used, we do not use that convention, but rather the one from \cite{Klingen}. Of course, the theorem shows that for number fields, the choice of terminology does not matter. 

In this paper, we consider what happens to this theorem if number fields are replaced by function fields. The notational set-up is as follows. Let $F$ denote a purely transcendental field $\F_q(T)$ over a finite field $\F_q$ with $q$ elements, and $A=\F_q[T]$ the corresponding polynomial ring. Assume that $K$ and $L$ are two subfields of a fixed algebraic closure of $F$, and assume that $K$ and $L$ are two geometric extensions of $F$. Let $N$ denote a finite Galois extension of $F$ that contains both $K$ and $L$. 

By slight but continuous abuse of notation, we will not indicate the dependence of zeta functions on $F$ and $A$ (which is up to finitely many Euler factors). 

There is the following classical theorem of Tate (\cite{Tate}) and Turner (\cite{Turner}):
\begin{thm} The fields $K$ and $L$ have the same Weil zeta function precisely if the Jacobians of the corresponding curves are $F$-isogenous. 
\end{thm}

Here, by the Weil zeta function of $K$ we mean $\zeta_K(s) = \sum_{I} {\mathbf N}(I)^{-s}, $ where we sum over non-zero ideals of the integral closure $\cO$ of $A$ in $K$, and ${\mathbf N}(I)=q^{\deg(I)}$.  Note that one may also sum over places; the corresponding zeta function is then the one above multiplied by the Euler factor `at infinity' $\prod_{i=1}^r (1-N(\infty_i)^{-s})^{-1}$, if $\infty_i$ are the places above $\infty=T^{-1}$. 

Mimicking the terminology above, we call $K$ and $L$ arithmetically, linearly or Ga{\ss}mann equivalent if the obvious function fields version of the corresponding statements in Theorem \ref{thm1} holds. We will use a general representation theoretical argument to show that Ga{\ss}mann equivalence implies equality of Weil zeta functions (Prop.\ \ref{esa}). We will show an example of function fields with the same Weil zeta function, but that are not Ga{\ss}mann equivalent (Ex.\ \ref{malakie}).

Now David Goss \cite{Goss} (following preliminary work of Carlitz) has introduced a new characteristic $p$ valued zeta function for function fields, essentially interpolating the function $$\zeta^{[p]}_K(s) = \sum_I \frac{1}{N(I)^s}$$ from positive integers $s$ to a function field analogue of the complex plane, which anyhow (only this is important for us) includes all integer values of $s$. Here, the sum ranges over nonzero ideals of the ring of integers $\cO_K$ (integral closure of $A$) of $K$, and the norm is from $\cO_K$ \emph{to $A$, so characteristic $p$ valued}, in contrast to the Weil zeta function. Again, one may incorporate infinite places.

The aim of this note is to investigate whether equality of Goss zeta functions is the same as Ga{\ss}mann-equivalence. We will show that Ga{\ss}mann, split, and arithmetical equivalence are the same for function fields (Prop.\ \ref{blub}). Concerning the relation to zeta functions, we prove

\begin{thm}  Let $K$ and $L$ denote two finite geometric field extensions of $F=\F_q(T)$. \begin{enumerate}
\item[\textup{(i)}] If $p>[K:F]$ and $p>[L:F]$, then $\zeta^{[p]}_K(s)=\zeta^{[p]}_L(s)$ means the same as arithmetical, split and Ga{\ss}mann-equivalence of $K$ and $L$ \textup{(cf.\ Thm.\ \ref{fuit})}.
\item[\textup{(ii)}] There exist infinitely many field extension $K$ of $F$ that have the same Goss zeta function, but are not arithmetically equivalent. Actually, as soon as $K$ is a Galois $p$-extension of $F$ only ramified above $\infty$, its Goss zeta function is $\zeta^{[p]}_K(s)=\zeta^{[p]}_F([K:F]s)$. In particular, the Goss zeta function does not determine the Weil zeta function of a function field \textup{(cf.\ Prop. \ref{funny} and Cor.\ \ref{for})}. 
\end{enumerate}
\end{thm}

The `problem' with the Goss zeta function is that it takes values in positive characteristic, so it only determines the splitting type `mod $p$'. We can fix this, however, by introducing a `Teichm\"uller lift' of the Goss zeta function, which we call the \emph{lifted Goss zeta function}.

\begin{df} The \emph{lifted Goss zeta function} of $K$ (w.r.t.\ $F$ and $A$) is 
$$\zeta_K^{[0]}(s) = \sum_{I} \frac{1}{\chi\left(N(I)^s\right)}, $$
where $\chi \, : \, F_\infty \, \rightarrow W$ is the Teichm\"uller lift from $F_\infty=\F_q((T^{-1}))$ to its ring of Witt vectors. 
\end{df}

This new zeta function has some obvious properties that it shares with the Goss zeta function, such as permiting an Euler product, but also some nice characteristic zero properties, such as subsiding to the Artin formalism of factorisation according to characters (cf.\ Section \ref{LG}).

We prove: 
\begin{thm} Let $K$ and $L$ denote two finite geometric field extensions of $F=\F_q(T)$.
\begin{enumerate}
\item[\textup{(i)}] 
 The equality $\zeta^{[0]}_K(s)=\zeta^{[0]}_L(s)$ means the same as arithmetical, split and Ga{\ss}mann-equivalence, irrespective of any bounds on the degrees of $K$ and $L$ over $F$ \textup{(cf.\ Thm.\ \ref{brr}).}
\item[\textup{(ii)}] The lifted Goss zeta function of a function field $K$ determines its Weil zeta function and its Goss zeta function \textup{(cf.\ Cor.\ \ref{crr}).} 
\item[\textup{(iii)}] For every $q$, there exist two function field extensions of $F=\F_q(T)$ such that $K$ and $L$ have the same lifted Goss zeta function, but are not isomorphic; the degree of $K/F$ and $L/F$ has to exceed $6$ for this to be possible. \textup{(cf.\ Prop.\ref{drr}).} 
\end{enumerate}
\end{thm}

\begin{notation} For the convenience of the reader, we gather our notational convention for the three different zeta functions of a function field $K$:
\begin{enumerate}
\item $\zeta_K$ is the Weil zeta function of $K$ --- which takes values in the complex numbers $\C$;
\item $\zeta^{[p]}_K$ is the Goss zeta function of $K$ --- the superscript indicates that this function takes values in characteristic $p$;
\item $\zeta^{[0]}_K$ is the lifted Goss zeta function of $K$ --- the superscript indicates that this function takes values in characteristic $0$; actually, in the ring of Witt vectors of $\F_q((T^{-1}))$.
\end{enumerate}

\end{notation}

\section{Ga{\ss}mann, linear and arithmetical equivalence for function fields}

Throughout, let $F$ denote a purely transcendental field $\F_q(T)$ over a finite field $\F_q$ with $q$ elements. Assume $K$ and $L$ are two subfields of a fixed algebraic closure of $F$, and assume that $K$ and $L$ are two geometric extensions of $F$. Let $N$ denote a finite Galois extension of $F$ that contains both $K$ and $L$. Let $G$ denote the Galois group of $N$ over $F$, and for any subfield $M$ of $N$,  let $H_M$ denote the Galois group of $N$ over  $M$. We have a diagram of function fields, and a corresponding diagram of curves:
$$
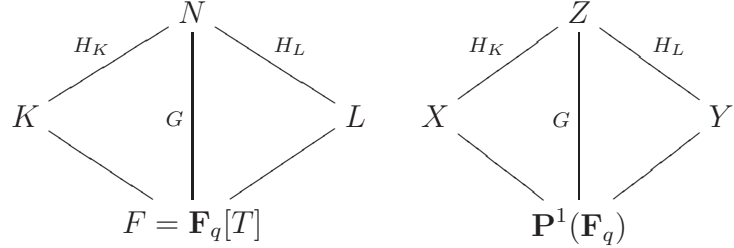
\xymatrix{ 
 & N  \ar@{-}[dl]_{H_K} \ar@{-}[dr]^{H_L} \ar@{-}[dd]_G &\\
K  \ar@{-}[dr]          & &     L \ar@{-}[dl] \\
 &                 F= \bF_q[T]                 &  }
\, \, \, \, \, \, \, 
\xymatrix{ 
 & Z  \ar@{-}[dl]_{H_K} \ar@{-}[dr]^{H_L} \ar@{-}[dd]_G &\\
X  \ar@{-}[dr]          & &     Y \ar@{-}[dl] \\
 &                 \P1^1(\F_q)                &  }
$$

We start with proving that the equivalence of the first three conditions in Theorem \ref{thm1} is preserved for global function fields. This is essentially because the proofs only depend on group theory (in particular, rational representation theory of a permutation representation), and general facts from algebraic number theory that persists in global fields. We nevertheless outline the proofs with references: 

\begin{prop} \label{blub}
The following are equivalent:
\begin{enumerate}
\item[\textup{(i)}] \textup{[Arithmetic Equivalence]} For every irreducible polynomial $\p$ in $F$ outside a set of Dirichlet density zero, the \emph{splitting type} of $\p$ (i.e., the ordered sequence of inertia degrees with multplicities) in $K$ and $L$ is the same;
\item[\textup{(i')}] \textup{[Split Equivalence]} For every irreducible polynomial $\p$ in $F$ outside a set of Dirichlet density zero, the number of prime ideals in $K$ and $L$ above $\p$ are the same;
\item[\textup{(ii)}] \textup{[Ga{\ss}mann Equivalence]} Every conjugacy class in $G$ meets $H_K$ and $H_L$ in the same number of elements.  
\end{enumerate}
\end{prop}

\begin{proof}
We note that (i) implies (i'), since the latter is weaker (only talking about the length of the splitting type). 

The implication (i') $\Rightarrow$ (ii) is proven by Perlis and Stuart in \cite{Perlis2} for number fields, and we outline the (similar) argument for function fields. The point is that the proof uses only general features of algebraic number theory that hold over global fields, and the theory of rational representations.  Let $S$ denote the set of Dirichlet density zero from the statement, enlarged by the finitely many ramified primes in $N/F$. For any $g \in G$, let $C$ denote the cyclic subgroup generated by $g$. By function field Chebotarev (e.g., \cite{Jarden} Thm.\ 6.3.1), we can find an irreducible polynomial $\p$ outside $S$ having decomposition group $C$ in $N$.   The splitting number of $\p$ in $K$ is the number of double cosets $|C\backslash G / H_K|$ and similarly for $L$ (the number field statement is due to Hasse, it is reproven as Lemma 1 in \cite{Perlis2}, and this proof works \emph{mutatis mutandis} for function fields). Let $\rho_K=\Ind_{H_K}^N 1$ and $\rho_L=\Ind_{H_L}^N 1$. Now $|C \backslash G / H_K|$ is the dimension of the $\rho_K(C)$-invariants of the $\mathbf{Q}[G]$-module $\mathbf{Q}[G/H_K]$, and similarly for $L$ (cf.\ Lemma 2 in \cite{Perlis2}).  The hypothesis hence implies that $\rho(C)$-invariant subspaces of $\mathbf{Q}[G/H_K]$ and $\mathbf{Q}[G/H_L]$ have the same dimensions for all cyclic subgroups $C \subseteq N$. By \cite{Serre2}, Cor.\ to Thm.\ 30, this implies that the representations $\rho_K$ and $\rho_L$ are isomorphic, and this is the same as Ga{\ss}mann equivalence.

For (ii) $\Rightarrow$ (i), Ga{\ss}mann has proved the purely group theoretical fact that if $H_K$ and $H_L$ are Ga{\ss}mann equivalent, then for every cyclic subgroup $C$ of $G$, the \emph{coset types} of $(G,H_K,C)$ and $(G,H_L,C)$ coincide. Here, the coset type of $(G,H,C)$ is the list of integers $(f_1,\dots,f_r)$ (in increasing order) defined by $f_i:=|H\tau_iC|/|H|$ if $G=\bigcup H \tau_i C$. It remains to show the number theoretical fact that the coset type of $(G,H_K,C)$ is the splitting type of $\p$, if $C$ is a (cyclic) decomposition group of (an unramified) $\p$. But this follows from the cited theorem of Hasse (see \cite{Perlis2}, Lemma 1, part (b) for the case of number fields, but the proof is the same for function fields). 
Now one can leave out a set of Dirichlet density zero, this only weakens the statement.   
\end{proof}

\section{Equivalences for the Weil zeta function}

Let us first present an example that explicitly shows the failure of the analogue of the equivalence of (iii) to the other parts of Theorem \ref{thm1} for function fields and the Weil zeta function:

\begin{example} \label{malakie} We show an explicit example where two function fields are not Ga{\ss}mann equivalent, but do have the same Weil zeta functions:
$$\xymatrix{ 
 & N= \bF_3(\sqrt{T},\sqrt{T+1})  \ar@{-}[dl]_{H_1} \ar@{-}[dr]^{H_2} \ar@{-}[dd]_G &\\
K= \bF_3(\sqrt{T})   \ar@{-}[dr]          & &     L= \bF_3(\sqrt{T+1}) \ar@{-}[dl] \\
 &                 F= \bF_3(T)                 &  }$$
Then $H_1=H_2=\bZ/2$ and $\Gal(N/F)=H_1 \times H_2$. If we pick $c\neq 1$ in $H_1$, then its conjugacy class intersects $H_1$ in $\{c\}$ and doesn't intersect $H_2$, so $K$ and $L$ are not Ga{\ss}mann-equivalent,  
but since $K$ and $L$ are isomorphic by the map sending the element $T$ to $T+1$, 
their zeta functions are equal. 
\end{example}

We now present an easy unified proof of the following statement for global fields:

\begin{prop} \label{esa}
If $K$ and $L$ are Ga{\ss}mann equivalent global fields, then they have the same Weil zeta function. 
\end{prop}

\begin{proof}
$K$ and $L$ are Ga{\ss}mann equivalent if and only if the representations $\Ind_{H_K}^G 1$ and $\Ind_{H_L}^G 1$ are isomorphic (e.g., \cite{Perlis2}, p.\ 301-302). By Kani and Rosen (\cite{KR}, p.\ 232, property (d)), we have an expression for the Weil zeta function (in the notation of Serre \cite{Serre} for general zeta functions of schemes over $\Z$): 
$$ \zeta_K(s) =  \zeta(N^{H_K},s) \cong L(Z,\Ind_{H_K}^G 1; s) \cong L(Z, \Ind_{H_L}^G 1; s) = \zeta(N^{H_L},s)=\zeta_L(s).$$
Here, $\cong$ means equality up to the infinite Euler factors, which match up anyhow since $K$ and $L$ are arithmetically equivalent by Proposition \ref{blub}.
\end{proof}

\section{Equivalences for the Goss zeta function}

An elementary explanation of the failure of the Weil zeta function to detect arithmetic equivalence is that it provides a counting function only for the number of ideals of a given \emph{degree} in a function field. Now David Goss in \cite{Goss} (cf.\ \cite{Gossbook}) has introduced a completely different zeta function for function fields, which is based on a norm that does not take values in $\Z$ (or rather $q^{\Z}$), but in the ring of integers of the ground function field $F$ itself. So this has some more hope of capturing the splitting type of primes, and investigating this situation is the aim of this section. 

Fix the place $\infty=T^{-1}$ of $F$ of degree one, let $A=\F_q[T]$ denote the corresponding ring of integers in $F$, and let $\cO:=\cO_K$ denote the integral closure of $A$ in $K$ (and similarly for $L$). Let $N$ denote the "norm map" from ideals of $\cO$ to ideals of $A$ defined as follows: if $\B$ is a prime ideal of $\cO$ lying above the monic ireducible polynomial $\p \in A$ given by $(\p)=\B \cap A$, then we set $N(\B):=\p^{\deg(\B)/\deg(\p)}$, and extend this multiplicatively to all ideals of $\cO$. Now Goss' zeta function is 
$$ \zeta_K^{[p]}(s):= \sum_{I} \frac{1}{N(I)^s}. $$
This is defined for $s$ in a certain "plane" $S_\infty = \C_F^* \times \Z_p$, where $\C_F$ is the completion of an algebraic closure of the completion $F_\infty$ of $F$ w.r.t.\ $T^{-1}$, and $\Z_p$ is the ring of $p$-adic integers. The Goss zeta function takes values in $F_\infty$. Note also that strictly speaking, the Goss zeta function depends on the choice of the place $\infty$, hence on the ring of integers $\cO$, but we will not indicate this in the notation.  For now, it is not important to know how to define the power $N(I)^s$ for such general $s$, it suffices for what follows to define it for $s \in \Z_{>0}$. Such functions were considered earlier by Carlitz (e.g., \cite{Carlitz}) and Goss (\cite{GossStaudt}). One can actually show the following: 

\begin{prop}[Goss, \cite{Gossbook}, p.\ 260] A general function field Dirichlet series $D(s):=\sum_{I}c(I)
N(I)^{-s}$ for some $c(I) \in F_\infty$, with nontrivial half plane of convergence in $s$, is uniquely determined by its values $D(j)$ at sufficiently large integers $j$.
\end{prop}

This proposition, of which the analogue for general complex valued Dirichlet series is well known (\cite{Hardy} Thm.\ 6), implies that in our situation, we have:

\begin{cor} For both number fields and function fields, and for Dirichlet, respectively for both the Weil and Goss zeta functions, we have that 
 the condition 
\begin{enumerate}
\item[\textup{(iii)}] $\zeta_K(s)=\zeta_L(s)$ for all $s$
\end{enumerate}
 is equivalent to the condition 
\begin{enumerate}
\item[\textup{(iii')}] $\zeta_K(j)=\zeta_L(j)$ for all sufficiently large integers $j$. 
\end{enumerate}
\end{cor}

\begin{thm} \label{fuit}
If $[L:F]<p$ and $[K:F]<p$, then equality of the Goss zeta functions of $K$ and $L$ is the same as arithmetic equivalence of $K$ and $L$. 
\end{thm}

\begin{proof}
We observe that the Goss zeta function can be written as 
\begin{equation} \label{g} \zeta_K^{[p]}(s) = \sum_{n \in A} \frac{B(n)}{n^s}, \end{equation}
where we sum over monic $n \in A$, and $B(n)$ is the number of ideals of $\cO_K$ of norm $n \in A$, but modulo $p$, since the value is in $F_\infty$. For an integer $f$ and an irreducible element $\p$ of $A$, we then know the integer $C(\p^f)$ modulo $p$, where $C(\p^f)$ is the number of \emph{prime} ideals of $\cO_K$ with norm $\p^f$. Indeed, doing the combinatorics (e.g., Klingen \cite{Klingen}, proof of Thm.\ 2.1 on p.\ 16), we find:
\begin{eqnarray} \label{g2} B(\p^f) &=& \sum_{{{r_1 f_1 + \dots r_s f_s = f}\atop{1 \leq f_1 < \dots < f_s}}\atop{1 \leq r_1, \dots, r_s}}  \prod_{i=1}^s {{C(\p^{f_i})+r_i-1} \choose{r_i}} \\ &=& C(\p^f) + [\mbox{ function of } C(\p^{f'}) \mbox { for } f'<f].          \end{eqnarray}
However, since $C(\p^f)<[K:F]<p$, we know the integer $C(\p^f)$ on the nose. Now knowing all integers $C(\p^f)$ implies knowing the splitting types of all $\p \in A$ (namely, for $1 \leq f \leq [K:F]$, put $f$ in the splitting type of $\p$ exactly $C(\p^f)$ times).
Hence equality of Goss zeta functions implies equality of splitting types. 
\end{proof}

\begin{rem}
The formula in \cite{Perlis1}, p.\ 346, expressing the corresponding $C$'s in terms of the $B$'s should be replaced by (\ref{g2}), but this does not effect the argument in \cite{Perlis1}, since it only uses that the $C$'s are polynomially expressible in terms of the $B$'s.
\end{rem}

However, if the degree condition in the above theorem is not met, then the conclusion is false. For example, we have: 

\begin{prop} \label{funny} Suppose that $K$ is a Galois extension of $F$ with Galois group $G$ a $p$-group, only ramified above $\infty$. Then the Goss zeta function of $K$ satisfies $$\zeta^{[p]}_K(s)=\zeta^{[p]}_F(|G|s).$$ In particular, equality of Goss zeta functions does not imply arithmetic equivalence in general, not even in the case of Galois extensions. 
\end{prop}

\begin{proof}
This is proven in \cite{Gossbook}, Remark 8.18.7.2 using Artin factorisation, but we present an elementary computation. The set of ideals of $\cO_K$ of norm $\p^f$ admits an action of $G$. The number of elements in the orbit of a given ideal $I$ is congruent to $0$ or $1$ mod $p$, depending on whether the action of $G$ on $I$ is nontrivial or trivial, respectively. Since the action of $G$ permutes the prime ideals above $\p$ transitively, if $I=\B_1^{a_1}\dots\B_s^{a_s}$ for some prime ideals $\B_i$ and some integers $a_i$, the action of $G$ on $I$ is only trivial if $I=(\B_1\dots\B_s)^r=(\p\cO_K)^r$ for some integer $r$. But then $N(I)=\p^{s|G|}$ so we conclude that $$B(\p^f)= \left\{ \begin{array}{ll} 0 \mbox{ mod } p & \mbox{ if } |G| \nmid f \\ 1 \mbox{ mod } p & \mbox{ if } |G| \mid f \end{array} \right.. $$ Now $B$ is multiplicative on coprime ideals, so we can compute  
\begin{eqnarray*}
\zeta^{[p]}_L(s) &=& \sum_{n \in A-\{0\}} \frac{B(n)}{n^s} = \prod_{\p} \sum_{k=0}^\infty \frac{B(\p^k)}{\p^{ks}} = \prod_{\p} \sum_{k=0}^\infty \frac{1}{\p^{|G|ks}} \\ &=& \prod_{\p} \frac{1}{1-\p^{-|G|s}}=\sum_n \frac{1}{n^{|G|s}} = \zeta^{[p]}_F(|G|s).
\end{eqnarray*}

For the last claim, just choose two $p$-extensions ramified above $\infty$ with different Galois groups of the same order, e.g.\ $\Z/p \times \Z/p$ versus $\Z/p^2$ given as $X^{p^2}-X=T$ and the corresponding equation in truncated Witt vectors, respectively. 
\end{proof}

\begin{rem}
The above proposition is really about the zeta functions \emph{without} the infinite (here, ramified) factors, as should also become clear from the proof. It is easy to include the correct factors at $\infty$ depending on the ramification, cf. also Section 8.10 in \cite{Gossbook}. 
\end{rem}

\begin{cor} \label{for}
The Goss zeta function of a function field does not determine the Weil zeta function of a function field in general.
\end{cor}

\begin{proof}
If we construct two $p$-group extensions of $F$ only ramified above $\infty$, with different genera but the same degree, we are done, since the Weil zeta function as rational function in $q^{-s}$ has a numerator of degree twice the genus of the extension. This is clearly possible, for example by varying the Hasse conductor (e.g., the cyclic $p$-extension $X^p-X=T^m$ has genus $\frac{1}{2}((mp-p-m+1)$, varying with $m$). \end{proof}

\begin{rem} \label{gossrem}
If we denote an exponent $s$ in the full Goss plane $S_\infty=\C_F^* \times \Z_p$ as $s=(x,y)$ with $x \in \C_F^*$ and $y \in \Z_p$, then the definition of $s$-exponentiation implies that $N(I)^{(x,0)}=x^{\deg(I)}$. Hence we can write
$$ \zeta_K^{[p]}((x,0)) = \sum_I x^{-\deg(I)} $$
as a function of $x \in \C_F^*$. 
On the other hand, the Weil zeta function is a function of $u=q^{-s}$:
$$ \zeta_K(s) = Z_K(u) = \sum_I u^{\deg(I)}. $$
So if we let the formal variable $u$ take values in $\C_F^*$, too, we find an identity
$$  \zeta_K^{[p]}((x,0)) = Z_K(x^{-1}) \mbox{ mod } p,$$
so the Goss zeta function encodes the Weil zeta function ``modulo $p$''. We thank David Goss for this remark. 

Since the numerator of $Z_K(u)$ is the characteristic polynomial of Frobenius, 
this implies that the Goss zeta function encodes this polynomial \emph{ mod } $p$. In particular, for an elliptic curve, the Goss zeta function does determine the Weil zeta function, since both are determined by the trace of Frobenius, which is bounded by $\sqrt{p}$ in absolute value. However, in higher genus, there are curves for which not even the trace of Frobenius is determined by its value mod $p$. 

Note that we have taken zeta functions relative to $A$ (so without the Euler factors at $\infty$), but this doesn't change the argument, one may just add these Euler factors and consider the `absolute' zeta functions. 
\end{rem}

\section{The lifted Goss zeta function, and equivalences} \label{LG}

The problem with the Goss zeta function is that the coefficients are 
in a field of characteristic $p$. Recall that $F_\infty = \F_q((T^{-1}))$ is the $T^{-1}$-adic completion of $F$. Let us denote by $W=W(F_\infty)$ the ring of Witt vectors of $F_\infty$. Let $\m$ denote its maximal ideal, and $\pi$ a fixed uniformizer.   Now $W$ is a ring of characteristic zero 
and there is a natural injection $\Z \rightarrow W$.
Let $\chi: F_\infty \rightarrow W$ denote the Teichm\"uller character, sending $F_\infty \ni x \mapsto 
\chi(x)=(x,0,\cdots)$. Recall that $\chi$ induces a group homomorphism on multiplicative groups. 
We define the \emph{lifted Goss zeta function} as
\begin{equation} \label{*zeta}
\zeta_K^{[0]}(x)=\sum_{I} \frac{1}{\chi\left(N(I)^s\right)}.
\end{equation}
Note that the multplicativity of $\chi$ implies that $$\zeta_K^{[0]}(x)=\sum_{I} \frac{1}{\chi\left(N(I)^s\right)} = \sum_{I} \chi \left( \frac{1}{N(I)^s} \right). $$ For integers $s$, we also have $$\zeta_K^{[0]}(x)=\sum_{I} \frac{1}{(\chi(N(I))^s},$$ and if we \emph{define} $s$-exponentiation on $x \in \chi(F_\infty)$ by $x=\chi(y) \mapsto \chi(y^s)$ (where $y \rightarrow y^s$ is Goss exponentiation), then the formula holds for any $s$ in the Goss complex plane. 
We further note that 
\begin{equation} \label{mod}
\zeta_K^{[p]}(s) \equiv \zeta_K^{[0]}(s) \mod \m.
\end{equation}
so that the lifted Goss zeta function contains all information that the Goss zeta function contains. Also, since $\chi \circ N$ from ideals of $\cO_K$ to $A$ is multiplicative, the lifted Goss zeta function has an Euler product
$$ \zeta^{[0]}_K(s) = \prod_{\B} \frac{1}{1-\chi(N(\B))^{-s}}. $$ 
Let $\Z \ni B(n)=\# \{ I: N(I)=n \}$ denote the number of ideals of $\cO_K$ of norm $n$, and let 
$$B(n)=\sum_{\nu=0}^\infty B_\nu(n)\pi ^\nu$$ be the $\pi$-adic expansion of $B(n)$.
Then we can write a `$\pi$-adic expansion' of the lifted Goss zeta function as follows
\begin{equation} \label{*zeta2}
\zeta_K^{[0]}(s)=\sum_{n \in A} \frac{B(n)}{\chi(n)^s}
=\sum_{\nu=0}^\infty  \pi^\nu \zeta_K^{p^\nu}(s), \mbox{ where }   \zeta_K^{p^\nu}(s):=\sum_{n\in A}
\frac{B_\nu(n)}{n^s} \in K_\infty, 
\end{equation}
again summing over monic $n \in A$. 

\begin{remarks}
\mbox{ }
\begin{enumerate}
\item[(i)] There is no problem with the convergence of this lifted Goss zeta function since it is a Witt vector of convergent sequences.
\item[(ii)] It is clear from (\ref{mod}) that $\zeta^{[0]}(s)=0$ implies $\zeta^{[p]}(s)=0$. In particular, the analogue of the Riemann hypothesis for $\zeta^{[p]}(s)$ (which seems to be that all `nontrivial' zeros in the Goss plane are in $K_\infty$, proven for rational function fields  in \cite{Wan}, \cite{Sheats}) implies its truth for $\zeta^{[0]}(s)$.  
\item[(iii)] The theory of `trivial zeros' and functional equation(s) of the Goss zeta function is very intriguingly nonclassical (\cite{Thakur}, \cite{Gossblogs}). Does the lifted Goss zeta function behave less mysteriously? Zeros of $\zeta_K^{[p]}$ need not extend to zeros of $\zeta_K^{[0]}$. We believe this point deserves more attention, especially in connection with simultaneously `lifting' the theory of Drinfeld modular forms and associated Galois representations to characteristic zero.  
\end{enumerate}
\end{remarks}

\begin{thm}\label{brr}
Let $K$ and $L$ be arbitrary geometrical extensions of $F=\F_q(T)$. The condition $\zeta^{[0]}_K(s)=\zeta^{[0]}_L(s)$ is the same as arithmetical equivalence of $K$ and $L$.
\end{thm}

\begin{proof}
Since 
$$ \zeta_K^{[0]}(s)=\sum_{n \in A} \frac{B(n)}{\chi(n)^s},$$
knowledge of the lifted Goss zeta function implies knowledge of the \emph{integers} (not just mod $p$) $B(n)$. By formula (\ref{g2}), this implies knowledge of the splitting type of all (unramified)  prime ideals of $K$. 
\end{proof}

\begin{cor} \label{crr} 
Knowledge of the lifted Goss zeta function implies knowledge of the Weil zeta function. 
\end{cor}

\begin{proof}
This follows from Proposition \ref{esa}. 
\end{proof}

\begin{rem}
Along the lines of the remark of David Goss in \ref{gossrem}, there is the following direct relation between the lifted Goss zeta function and the Weil zeta function:
$$ Z_K(u)=\zeta^{[0]}_K((\chi^{-1}(u),0)), $$
where $u \in W$. 
\end{rem}

\begin{rem} The formalism of Artin of factorisation of zeta functions of a Galois extension according to characters of the Galois group applies to the lifted Goss zeta function (the situation is unclear for the (unlifted) Goss zeta function). Since we have chosen $\chi$ to be the multiplicative Teichm\"uller lift, $\chi \circ N$ is multiplicative, hence $\zeta_K^{[0]}$ admits an Euler product. For $K/F$ Galois with Galois group $G$, let $(V,\rho_G)$ denote the regular representation of $G$ over an algebraic closure of the fraction field of $W$, then we have
$$ \zeta_K^{[0]}(s) = \prod_{\B} \det(1-\rho(\mbox{Frob}_{\B}) \chi(N(\B)^{-s}) \, | \, V^{I_{\B}}), $$
product over prime ideals $\B$ of $\cO_K$. 
Let $L/K$ denote a Galois extension with Galois group $G_{L/K}$, then 
$$ \zeta_L^{[0]}(s) = \prod_{\rho} L_K^{[0]}(\rho,s)^{\chi_\rho(1)}, $$
for $\rho$ running through the irreducible representations of $G_{L/K}$, and where the \emph{lifted Goss $L$-series} $L_K^{[0]}(\rho,s)$ is defined as 
$$ L_K^{[0]}(\rho,s)=\prod_{\B}  \det(1-\rho(\mbox{Frob}_{\B}) \chi(N(\B)^{-s}) \, | \, V^{I_{\B}}). $$
 An application of this formalism is a conceptual proof of Prop.\ \ref{funny}, as sketched in \cite{Gossbook}, p.\ 296. 
\end{rem}

Finally, we prove

\begin{prop} \label{drr}\mbox{}
\begin{enumerate}
\item[\textup{(i)}] Let $F=\F_q(T)$ for any prime power $q$. There exist two geometric extensions $K$ and $L$ over $F$ such that $\zeta_K^{[0]}(s)=\zeta_L^{[0]}(s)$ (so $K$ and $L$ are arithmetically equivalent), but $K$ and $L$ are not isomorphic as field extensions of $F$.
\item[\textup{(ii)}] Let $F=\F_q(T)$ with $q$ a square if $p \equiv 2 \mbox{ mod } 3$. There exist two geometric extensions $K$ and $L$ of degree $7$ over $F$ such that $\zeta_K^{[0]}(s)=\zeta_L^{[0]}(s)$ (so $K$ and $L$ are arithmetically equivalent), but $K$ and $L$ are not isomorphic as field extensions of $F$. A counterexample in smaller degree does not exist. 
\end{enumerate}
\end{prop}

\begin{proof}

For (i), we pick up the construction of Komatsu (\cite{Komatsu}): Let $\ell>2$ denote a prime number, let $H_1 = (\Z/\ell)^3$ and $$H_2= \left\{ \left( \begin{array}{ccc} 1 & \ast & \ast \\ 0 & 1 & \ast \\ 0 & 0 & 1 \end{array} \right) \in \mathrm{GL}(3,\F_\ell) \right\}.$$ Then $H_1$ and $H_2$ are Ga{\ss}mann equivalent  as subgroups (via the Cayley map) of the symmetric group $S_{\ell^3}$. Note that $H_2$ is not commutative for $\ell>2$, so in particular, not isomorphic to $H_1$. 

Now $S_{\ell^3}$ is realized over the field $\F_q(x_1,\dots,x_{\ell^3})$ as Galois group of the general polynomial, and since $\F_q(x_1)$ is Hilbertian, Hilbert's irreducibility (cf.\ \cite{Jarden} Prop.\ 16.1.5) implies that there exists a specialisation of $x_2,\dots,x_{\ell^3}$ to $\F_q(x_1)$ that gives a (geometrical) $S_{\ell^3}$ Galois extension of $\F_q(T)$ (setting $T=x_1$).

The extensions $K$ and $L$ corresponding to this construction are huge, of degree $(\ell^3-1)!$ over $F$ (a $27$-digit number for $\ell=3$). To construct an example of much smaller degree, we rely on modular curves, but at the cost of assuming $q$ to be a square: 
  
We prove (ii) as follows. The argument in \cite{Perlis1}, Thm.\ 3 is purely group theoretical, and shows that there are no Ga{\ss}mann equivalent subgroups in $S_n$ for $n \leq 6$, hence also no non-trivially Ga{\ss}mann equivalent function fields over $F$ of degree $\leq 6$. 

We know from \cite{Perlis1}, p.\ 358 that  $G=\PSL(2,7)$ has two Ga{\ss}mann equivalent subgroups of index $7$, so it remains to realize $G$ as Galois group over $F$. 

For this, first assume $p \neq 7$. Then let $N$ denote the function field of the Klein quartic over $\F_p$, a.k.a.\ the modular curve $X(7)_{/\Z}$ reduced modulo $p$. We know that $\Gal(N\bar{\F}_q/F\bar{\F}_q)$ contains $\PSL(2,7)$ (actually, is equal to it unless $p=3$, cf.\ \cite{Guralnick},  \cite{Elkies}, when it is a group of order $6048$ instead of $168$). However, the action of $\PSL(2,7)$ is only defined over $\F_q(\zeta_3)$ with $\zeta_3$ a third root of unity. Now if $p \equiv 1 \mbox{ mod } 3$, then $-3$ is a square mod $p$, so $\zeta_3 \in \F_q$, and in the other case, we assume $F \supseteq \F_{p^2}$, so we are fine, too. We don't know whether $K/F$ and $L/F$ can be descended further to $\F_p$ in general.  

Now if $p=7$, we use instead the result of Abhyankar (\cite{Abhyankar}) that the polynomial $X^8-TX+1$ has Galois group over $\bar{\F}_q F$ equal to $\PSL(2,7)$. One may compute directly, by hand or by \textsf{Magma} \cite{MAGMA} that in fact the automorphism group is defined already over $\F_7$. 
\end{proof}

\bibliographystyle{amsplain}
%\bibliography{arithequiv}
\providecommand{\bysame}{\leavevmode\hbox to3em{\hrulefill}\thinspace}
\providecommand{\MR}{\relax\ifhmode\unskip\space\fi MR }
% \MRhref is called by the amsart/book/proc definition of \MR.
\providecommand{\MRhref}[2]{%
  \href{http://www.ams.org/mathscinet-getitem?mr=#1}{#2}
}
\providecommand{\href}[2]{#2}

\end{document}